\documentclass[10pt]{amsart}
\usepackage{graphicx}
\usepackage{latexsym}
\usepackage{fancyhdr}
\usepackage{amsmath, amssymb}
 \usepackage[utf8]{inputenc}
\usepackage[all]{xy}
\usepackage{hyperref}
\usepackage{subfigure}
\usepackage{enumerate}
\usepackage{color}
\usepackage[normalem]{ulem}

\theoremstyle{plain}
\newtheorem{thm}{Theorem}
\newtheorem{cor}[thm]{Corollary}

\newtheorem{lem}[thm]{Lemma}

\theoremstyle{definition}

\newtheorem{rem}[thm]{Remark}
\numberwithin{equation}{section}

\newcommand{\lgw}{\longrightarrow}
\newcommand{\lgm}{\longmapsto}

\newcommand{\ovl}{\overline}
\newcommand{\lt}{\operatorname{LT}}

\newcommand{\x}{\mathbf{x}}
\newcommand{\y}{\mathbf{y}}

\newcommand{\Supp}{\operatorname{Supp}}
\newcommand{\D}{\Delta}

\newcommand{\xxi}{\mathbf{t'}}
\renewcommand{\P}{\mathbb P}
\newcommand{\la}{\lambda}
\renewcommand{\O}{\mathcal{O}}

\newcommand{\Z}{\mathbb{Z}}

\renewcommand{\k}{\Bbbk}

\newcommand{\Res}{\operatorname{Res}}

\newcommand{\R}{\mathbb{R}}
\newcommand{\K}{\mathbb{K}}

\newcommand{\N}{\mathbb{N}}

\renewcommand{\L}{\mathbb{L}}

\newcommand{\C}{\mathbb{C}}

\newcommand{\Q}{\mathbb{Q}}

\renewcommand{\tt}{\mathbf{t}}

\renewcommand{\lg}{\langle}
\newcommand{\rg}{\rangle}
\newcommand{\q}{\mathbf{q}}
\renewcommand{\a}{\alpha}
\renewcommand{\b}{\beta}
\newcommand{\g}{\gamma}

\renewcommand{\phi}{\varphi}
\renewcommand{\d}{\delta}

\newcommand{\e}{\varepsilon}
\newcommand{\z}{\mathbf{z}}

\newcommand{\FAGN}{\operatorname{QA}}

\begin{document}
\title[Varieties are homeomorphic to  varieties defined over number fields]{Algebraic varieties are homeomorphic to  varieties defined over number fields}

\author{Adam Parusi\'nski}
\email{adam.parusinski@unice.fr}
\address{Universit\'e Nice Sophia Antipolis, CNRS, LJAD, UMR 7351, 06108 Nice, France}

\author{Guillaume Rond}
\email{guillaume.rond@univ-amu.fr}
\address{Aix Marseille Univ, CNRS, Centrale Marseille, I2M, Marseille, France}

\thanks{Research partially supported by ANR project LISA (ANR-17-CE40-0023-03)}

\begin{abstract}
We show that every affine or projective algebraic variety defined over the field of real or complex numbers is homeomorphic to a variety defined over the field of algebraic numbers.  We construct such a homeomorphism by carefully choosing a small deformation of the coefficients of the original equations. This deformation 
preserves all polynomial relations over $\Q$ satisfied by these coefficients and 
is equisingular in the sense of Zariski.    

Moreover we construct an algorithm, that, given a system of equations 
defining a variety $V$, produces a system of equations with coefficients in $\ovl\Q$ 
of a  variety homeomorphic to $V$. 
\end{abstract}

\maketitle

\section{Introduction}

In computational algebraic geometry one is interested in computing the solution set of a given system of polynomial equations, or at least in computing various algebraic or geometric invariants of such a solution set. Here by computing we mean producing an algorithm that gives the desired result when it is implemented in an appropriate computer system.

Such an algorithm should be stable with respect to a small perturbation of the coefficients.  But in general the shape of the solution set may change drastically  under a small perturbation of the coefficients.  
This difficulty is particularly apparent if one has to deal with polynomial equations whose coefficients are neither rational nor algebraic numbers. 
The main goal of this paper is to show that, when we are interested in the topology of such a solution set, one can always assume that the polynomial equations have algebraic number coefficients precisely by choosing in an effective way a particular perturbation of its coefficients.

\begin{thm}\label{main}
Let $V\subset \K^n$ (resp. $V\subset \P_\K^n$) be an affine (resp. projective) algebraic set, where $\K=\R$ or $\C$. Then there exist an affine (resp. projective) algebraic set $W\subset \K^n$ (resp. $W\subset \P_\K^n$) and a homeomorphism $h:\K^n\lgw \K^n$ (resp. $h:\P_\K^n\lgw \P_\K^n$) such that:
\begin{enumerate}
\item[(i)] the homeomorphism maps $V$ onto $W$,
\item[(ii)] $W$ is defined by polynomial equations with coefficients in $\ovl\Q\cap \K$.
\item[(iii)] The homeomorphism $h$ is semialgebraic and arc-analytic.
\end{enumerate}
\end{thm}

Arc-analytic semialgebraic homeomorphisms play an important role in real algebraic geometry, see Remark \ref{rem:arc-analytic}.

\begin{rem}\label{more} In fact we prove a more precise and stronger result.  See Theorem \ref{main2} for a precise statement.
\end{rem}

In the proof of Theorem  \ref{main2} we show that $W$ can obtained by a small deformation of the coefficients of the equations defining $V$.  Let denote these coefficients by $g_{i,\alpha}$.  The deformation is denoted by $t\mapsto g_{i,\a}(t)$ with $g_{i,\a}(0)=g_{i,\a}$. This deformation is constructed  in such a way that every polynomial relation with rational coefficients satisfied by the $g_{i,\alpha}$ is satisfied by the $ g_{i,\a}(t)$ for every $t$, c.f. Theorem \ref{thm_alg}. 
So we can prove that the deformation is equisingular in the sense of Zariski since this equisingularity condition can be encoded in term of the vanishing and the non-vanishing of polynomial relations on the $g_{i,\a}(t)$. In particular this deformation is topologically trivial.

The first part of this paper is devoted to the proof of Theorem \ref{thm_alg}, an approximation result similar to the one given in \cite{Ro}. 
In the next part  we recall the notion of Zariski equisingularity. 
The proof of Theorem \ref{main2} is given in the following part. Finally we provide  an algorithm that, given the equations defining $V$, computes the equations defining $W$. This algorithm is based on the proof of Theorem \ref{main2}.

\subsection{Topologically trivial stratifications}
Theorem \ref{main} can be also proven by the following general argument. 
Denote, as above, the coefficients of the polynomials $F_1$, \ldots, $F_r$
 defining $V$ by $g_{i,\alpha}$ and denote their number by $N$. 
Let $\mathcal V \subset \K^N\times \K^n$ be the zero set of  
$F_1$, \ldots, $F_r$, where we consider their coefficients as parameters  
$(c_{i,\alpha})\in\K^ N$.   Thus $V$ is the fiber of $\mathcal V$ over  
$(g_{i,\alpha})$. 
There exists a stratification $\mathcal S $ of the coefficients space $\K^N$ 
such that the fibres of $\mathcal V$ are locally topologically trivial over each stratum.  Such a stratification can be, for instance, constructed by Zariski equisingularity, and in this case, by construction, each stratum 
$S$ of $\mathcal S$ is described by polynomial equations and inequalities in $c_{i,\alpha}$ with coefficients in $\Q$. 
Therefore, Theorem \ref{main} follows from the fact that $\ovl \Q\cap S$ 
is dense in $S$ (this follows, for instance, from our Theorem \ref{thm_alg}).  It seems that this argument was known to specialists in equisingularity theory, though we couldn't have found it in literature. 

In the real algebraic (or semialgebraic) set-up one can put the above argument 
in terms of Tarski's transfer principle, cf. \cite{BCR} Ch. 5.  More precisely, 
let us first work over the real closed field of algebraic numbers 
$\R_{alg} = \ovl \Q\cap \R$.  
By Hardt's trivialization theorem \cite{H} there is a stratification of 
$\mathcal S_{alg}$ of $\R^N_{alg}$ that trivializes $\mathcal V$. 
Moreover, the trivialization maps are semialgebraic.   By Tarski's 
transfer principle both the stratification and the trivializations extend to $\R$, 
that is the same equations and inequalities give a stratification   
$\mathcal S$ of $\R^N$, that is locally trivial with the trivializations given exactly by the same formulae.

Our approach is slightly different since we do not stratify the coefficient space.  Instead, as we explained above, we produce a deformation of the coefficients, that preserves polynomial relations with rational 
coefficients and therefore is topologically trivial.  For the last argument 
we use Zariski equisingularity, since it works over both $\R$ and $\C$, and 
gives trivialization with many additional properties (including semialgebraicity).
This deformation is based directly on the properties of the 
  field extension of $\Q$ generated by the coefficients $g_{i,\alpha}$.  
 For instance, it is particularly simple if this extension is purely transcendental, see  Remark \ref{rem_deg}. 

 \subsection*{Acknowledgements.} 
We would like to thank an anonymous referee for drawing our attention to the role of transfer principle in this context. 

 \subsection*{Notation.}  For a vector of indeterminates $x=(x_1,\ldots, x_n)$,   $x^i$ denotes the vector of indeterminates $(x_1,\ldots, x_i)$. 
To avoid confusion, variables will be denoted by normal letters $t$ and elements of $\K$ will be denoted by bold letters $\tt$.\\
For a field $\K$ the ring of algebraic power series is denoted by $\K\lg x\rg$.

\section{An approximation result}

Let  $\Omega \subset \K^{r}$  be open and non-empty  and let $\varphi$ be an analytic function on $\Omega$. 
We say that $\varphi$ is a \emph{Nash function} at $a \in\Omega$ if there exist an open neighborhood $U$ of $a$ in $\Omega$
and a nonzero polynomial $P\in \K[t_1,\dots,t_r,z]$ such that $P(\tt,\varphi (\tt))=0$ for $\tt\in U$ or, equivalently, if the Taylor series 
of $\varphi $ at   $a$ is an algebraic power series.  An 
analytic function on $\Omega$ is a Nash function if it is a Nash function at every point
of $\Omega$. An analytic mapping $\varphi:\Omega\to\K^N$ is a \emph{Nash mapping} if each of its components is a Nash function. Note that if $\Omega$ is connected and $\varphi$ is Nash then, by analytic continuation, it satisfies one polynomial equation 
 $P(\tt,\varphi (\tt))=0$ everywhere on $\Omega$.

We call a Nash function $\varphi : \Omega \to \K$,  $\Q$-algebraic (and we note $\phi\in\FAGN(\Omega)$) if (locally) it is algebraic over 
$\Q[t]$, i.e. it satisfies a polynomial relation $P(\tt,\varphi (\tt))=0$ for every $\tt\in\Omega$ with $P\in \Q[t_1,\dots,t_r,z]$, $P\neq 0$.  
It is well known, in the case when $\Omega$ is connected, that this means that the Taylor expansion of $\phi$ at a point with coordinates in $\Q$ (or $\ovl\Q$) is an algebraic power series whose coefficients lie in a common finite field extension of $\Q$, i.e. this Taylor expansion belongs to $\k\lg x\rg$ where $\k$ is a subfield of $\K$ which is finite over $\Q$ (see \cite{RD} for example).

\begin{rem} Usually, Nash functions are defined as real analytic functions defined on an open set $\Omega\subset \R^r$ whose graph is semialgebraic. In particular, $\Omega$ has to be semialgebraic. Here, we need a definition valid both for $\K=\R$ and $\C$. Moreover, in practice, the open set $\Omega$ will always be connected, and one may choose $\Omega$ semialgebraic.  Nash functions, as we defined above, are used sometimes in complex algebraic geometry, see for instance \cite{Lempert95}. Therefore, we have chosen a 
simple definition, that is slightly more general than the usual one in the real case, but easier to handle when we consider simultaneously the cases $\K=\R$ and $\C$. 
\end{rem}

\begin{thm}\label{thm_alg} Let $\K=\R$ or $\C$. 
Let $\y\in\K^m\setminus  \ovl\Q^m$. 
Denote by $\k$  the field extension of $\Q$ generated by  the components $\y_i$ of $\y$.  The field $\k$ is a primitive extension of a transcendental finitely generated field extension of $\Q$: $\k=\Q(\tt_1,\ldots, \tt_r)(\z)$ where the $\tt_i\in\K$ are algebraically independent over $\Q$ and $\z\in\K$ is finite of degree $d$ over $\Q(\tt_1,\ldots,\tt_r)$.\\
Then there exist  an open, connected and non-empty neighborhood of 
$\tt=(\tt_1, . . . , \tt_r)$, $\mathcal U 
\subset \K^r$,  a vector function well defined on $\mathcal U\times \K$ of the form 
$$y(t,z)\in \Q(t)[z]^m,$$
 and a scalar function $z(t)\in\FAGN(\mathcal U)$ such that 
$$\z = z( \tt ), \, \y=y(\tt,\z), $$ 
and for every $f(y)\in \Q [y]^p$, where $y=(y_1,\ldots, y_m)$ is a vector of indeterminates, such that $f(\y)=0$, we have 
$$f(y(t,z))=0.$$
Moreover, the function $z(t)$ is algebraic of degree  $d$ over $\Q(t)$.
\end{thm}

\begin{proof}[Proof of Theorem \ref{thm_alg}]
We denote, as in the statement of the theorem, by $\k$ the field extension of $\Q$ generated by  the components of $\y$, so $\k\subset \K$.
Since there are finitely many such coefficients, $\k$ is a finitely generated field extension of $\Q$. Let $\tt_1$, \ldots, $\tt_{r}$ be a transcendence basis of $\k$ over $\Q$ (with $r\geq 1$ because $\y\notin \ovl\Q^m$) and set $\L:=\Q(\tt_1,\ldots,\tt_{r})$. By the primitive element theorem, there exists an element $\z\in\k$ finite over $\L$ and such that $\L(\z)=\k$.\\
For every $i=1,\ldots, m$ we can write
$$\y_i=\sum_{k=0}^{d-1}\frac{p_{i,k}(\tt_1,\ldots,\tt_{r})}{q_{i,k}(\tt_1,\ldots,\tt_{r})}\z^k$$
where $d$ is the degree of $\z$ over $\L$, $p_{i,k}(t)$, $q_{i,k}(t)\in\Q[t]$ and $q_{i,k}(\tt_1,\ldots,\tt_{r})\neq 0$.  By multiplying the $p_{i,k}$ by some well-chosen polynomials we may assume that all the $q_{i,k}(t)$ are equal, let us say to $q(t)$.\\
\\
Let $P(t,z)\in\Q(t)[z]$ be the monic polynomial of minimal degree in $z$ such that
$$P(\tt_1,\ldots,\tt_{r},\z)=0.$$
Let $D\subset \K^{r}$ be  the discriminant locus of $P(t,z)$ seen as a polynomial in $z$ (i.e. $D$ is the locus of points $a\in\K^{r}$ such that $a$ is a pole of one of the coefficients of $P$ or such that $P(a,z)$ has at least one multiple root). Since $P(\tt_1,\ldots,\tt_{r},z)$ has no multiple roots in  an algebraic closure of $\L$, the point $\tt$ is not in $D$.  Then there exist   $\mathcal U\subset \K^{r}\backslash D$  a simply connected open neighborhood of $\tt$ and analytic functions
$$w_i : \mathcal U\lgw \K, \ \ i=1,\ldots, d$$
such that
$$P(t,z)=\prod_{i=1}^d(z-w_i(t))$$
and $w_1(\tt_1,\ldots,\tt_{r})=\z$.\\
Moreover  the $t\lgm w_i(t)$ are algebraic functions over $\Q[t]$. Let us set $\Q_\K=\Q$ when $\K=\R$ and $\Q_\K=\Q+i\Q$ when $\K=\C$. Then $w_1\in\FAGN(\mathcal U)$ and the Taylor series of $w_1$ at a point of $\mathcal U\cap\Q_\K^{r}$  is an algebraic power series whose coefficients belong to a finite field extension of $\Q$.
 \\
Since the polynomial $q$ is not vanishing at $\tt$ the function 
$$t\in\K^{r}\backslash\{q=0\}\lgm \frac{1}{q(t)}$$ is a Nash function whose Taylor series at a point of $\Q_\K^{r}\backslash\{q=0\}$ is an algebraic power series with rational  coefficients.\\
Let $b:=(b_1,\ldots,b_{r})\in \Q_\K^{r}\cap\mathcal U\backslash\{q=0\}$.  Let us denote by  $\phi_1(t)$  the Taylor series of $w_1$. Then let $\phi_2(t)\in\Q\lg t\rg$ denote the Taylor series of  $t\lgm \frac{1}{q(t)}$  at $b$. For simplicity we can make a translation and assume that $b$ is the origin of $\K^{r}$. \\
For every $i=1,\ldots, m$ let us define
\begin{equation}\label{y}y_i(t_1,\ldots, t_r,z,v):=v\sum_{k=0}^{d-1}p_{i,k}(t_1,\ldots,t_{r})z^k\end{equation}
and
$$y(t,z,v):=\left(y_1(t,z,v),\ldots, y_m(t,z,v)\right).$$ 
Let $f(y)\in \Q [y]^p$  such that $f(\y)=0$. We have that
$$f(y(\tt_1,\ldots,\tt_{r},\z,\frac{1}{q(\tt)}))=0$$
or equivalently
$$f(y(\tt_1,\ldots,\tt_{r},\phi_1(\tt_1,\ldots,\tt_{r}),\phi_2(\tt_1,\ldots,\tt_{r})))=0.$$
But the function
$$t\lgm f(y(t_1,\ldots,t_{r},\phi_1(t_1,\ldots,t_{r}),\phi_2(t_1,\ldots,t_{r})))$$ is an algebraic function over $\Q[t ]$ and $\tt_1$, \ldots, $\tt_{r}$ are algebraically independent over $\Q$. Thus we have that
$$f\big(y(t_1,\ldots,t_{r},\phi_1(t_1,\ldots,t_{r}),\phi_2(t_1,\ldots,t_{r}) )\big)=0.$$
Indeed, write $f=(f_1, \ldots, f_p)$ and for each $i=1,\dots , m$ consider the complex valued function 
$$\Phi_i (t) = f_i\big(y(t_1,\ldots,t_{r},\phi_1(t_1,\ldots,t_{r}),\phi_2(t_1,\ldots,t_{r}) )\big).$$ 
Note that $\Phi_i\in\FAGN(\mathcal U)$ because so is $w_1$.  Let $P_{1,i}\in \Q[t_1,\dots,t_r,z]$ be a polynomial of 
minimal degree such that $P_{1,i}(t,\Phi_i (t))=0$.  Note that $P_{1,i}$ is irreducible.  Write 
$$P_{1,i}(t,z) = \sum^{\deg P_{1,i}} _{k=0} a_{i,k}(t) z^k .$$
Then $a_{i,0}(\tt ) =0$ since $\Phi_i(\tt )=0$.  Therefore $a_{i,0} (t)=0$ because $\tt_1$, \ldots, $\tt_{r}$ are algebraically independent over $\Q$.   Hence $P_{1,i}=P_{2,i} z$ for some polynomial $P_{2,i}$, and $P_{2,i}$ is a unit since $P_{1,i}$ is irreducible. Therefore $\Phi_i (t)$ vanishes identically on $\mathcal U$.  \\
This proves the theorem by defining
$$y(t,z)=y(t,z,\phi_2(t))$$
and $z(t)=w_1(t)$. 
Indeed the coefficients of the components of $y(t,z,\phi_2(t))$ seen as polynomials in $z$ are rational power series by \eqref{y}. Moreover $\phi_1(t)$ belongs to a finite field extension  of $\Q(t)$ of degree $d$ since $P(t,z)$ is irreducible. 
\end{proof}

The following lemma will be used in the proof of the main theorem:

\begin{lem}\label{alg}
Let $\phi(t)\in\C\lg t\rg$ be a power series   algebraic  over $\Q[t]$ and $P(t,\Phi)\in\Q[t,\Phi]$ be a nonzero polynomial of degree $d$ in $\Phi$ such that $P(t,\phi(t))=0$. Let $\tt\in\Q^r$ be such that $\tt$ is not in the vanishing locus of the coefficients of $P(t,\Phi)$ seen as a polynomial in $\Phi$, and such that $\tt$ belongs to the domain of convergence of  $\phi(t)$. Then $\phi(\tt)$ is an algebraic number of degree $\leq d$ over $\Q$.
\end{lem}

\begin{proof}
The proof is straightforward: just replace $t$ by $\tt$ in the relation $P(t,\phi(t))=0$.
\end{proof}

\section{Zariski equisingularity}

Zariski equisingularity of families of singular varieties was introduced by 
Zariski in \cite{Za} (originally it was called the algebro-geometric equisingularity).  Answering a question of Zariski, Varchenko showed 
\cite{Va, Va73,VaICM} that Zariski equisingular families are locally topologically trivial.  
In the papers \cite{Va73,VaICM} Varchenko considers the families of local singularities while the paper \cite{Va} deals with the families of affine or projective algebraic varieties.  A new method of proof of topological triviality, giving much stronger statements, was given recently in \cite{PP}. 
The case of families of algebraic varieties was considered in sections 5 and 9 of \cite{PP}.  The version presented below follows from the proof of the main theorem, Theorem 3.3, of \cite{PP}, see also  Theorems 3.1  and 4.1 of \cite{Va} in the complex case and Theorems 6.1 and 6.3 \cite{Va} in the 
real case, and Proposition 5.2 and Theorem 9.2 of \cite{PP} where the algebraic global case is treated.

\begin{thm}\cite{Va},\cite{PP} \label{equisingularity}
Let $\mathcal V$ be an open connected neighborhood of $\tt$ in $\K^r$ and let $\mathcal O_{\mathcal V}$ denote the ring of $\K$-analytic functions on 
$\mathcal V$.  
Let $t=(t_1, \ldots , t_r)$ denote the variables in $\mathcal V$ and let $x=(x_1,\ldots, x_n)$ be a set of variables in $\K^n$. Suppose that for $i=k_0,\ldots, n,$ there are given 
$$
F_i(t,x^i)=x_i^{d_i}+\sum_{j=1}^{d_i}a_{i-1,j}(t,x^{i-1})x_i^{d_i-j}\in 
\mathcal O_{\mathcal V}[x^i], $$
with $\ d_i>0$,    such that

\begin{enumerate}
\item[(i)] for every $i>k_0$, the first non identically equal to zero generalized discriminant of $F_i(t,x^{i-1},x_i)$ 
equals  $F_{i-1} (t,x^{i-1})$ up to a multiplication by a nowhere vanishing function of $\mathcal O_{\mathcal V}$. 

\item[(ii)] the first non identically zero generalized discriminant of $F_{k_0}$ is independent of $x$ and does not vanish on $\mathcal V$.

\end{enumerate}
Let us set  for every $\q \in \mathcal V$, $V_\q=\{(\q,\x)\in \mathcal V \times \K^n\mid F_n(\q,\x)=0\}$.

Then  for every $\q \in\mathcal V$ there is a homeomorphism
$$h_\q: \{\tt\}\times\K^n\rightarrow \{\q\}\times\K^n$$ such that
$h_\q (V_\tt)=V_q.$

Moreover if $F_n=G_1\cdots G_s$ then for every $j=1,\ldots, s$  
$$h_\q\left(G_j^{-1}(0)\cap(\{\tt\}\times\K^n)\right)=G_j^{-1}(q)\cap(\{0\}\times\K^n).$$
\end{thm}

For the notion of generalized discriminants see subsection \ref{gd}.

\begin{rem}\label{proj}
The homeomorphism $h_\q$ of Theorem \ref{equisingularity} can be written $h_\q(\tt,x)=(\q,\Psi_\q(x))$, i.e., as a family of homeomorphisms $\Psi_\q :\K^n \to \K^n$.
If $F_i$ are homogeneous polynomials in $x$, then the homeomorphisms $\Psi_\q$ satisfy, by construction, 
$$\forall \la\in\K^*,\forall x\in \K^n\ \ \ \Psi_\q(\la x)=\la \Psi_\q(x)$$
Hence if we define $\P (V_\q)=\{(\q,\x)\in \mathcal V \times \P_\K^n\mid F_n(\q,\x)=0\}$, the homeomorphism $h_\q$ induces an homeomorphism between $\P(V_\tt)$ and $\P(V_\q)$.
\end{rem}

\begin{rem}
By construction of \cite{Va, Va73,VaICM} and \cite{PP}, the homeomorphisms $h_\q$ can be obtained by a local topological trivialization. That is there is a neighborhood $\mathcal W$  of $\tt$ in $\K^r$ and a homeomorphism 
$$
\Phi: \mathcal W \times \K^n \to \mathcal W \times \K^n
$$
such that $\Phi(\q,x)= (\q, h_\q (x)) $.  Moreover, as follows from \cite{PP}, 
we may require that:  
\begin{enumerate}
	\item
	The homeomorphism $\Phi$ is subanalytic.  In the algebraic case, 
	that is in the case considered in this paper, we replace in the assumptions  
	$\mathcal O_{\mathcal V}$ by the ring of $\K$-valued Nash functions on 
$\mathcal V$.  Then $\Phi$ can be chosen semialgebraic. 
\item
$\Phi$ is arc-wise analytic, see Definition 1.2 of \cite{PP}. In particular each $h_\q$ is arc-analytic.  It means that for every real analytic arc 
$\gamma (s) : (-1,1) \to \K^n$, $h_\q\circ \gamma$ is real analytic and the same property holds for $h_\q^{-1}$.
\end{enumerate}
\end{rem}


\section{Generic linear changes of coordinates}\label{generic_change}
Let $f\in\K[x]$ be a polynomial of degree $d$ and let $\k$ be a field extension of $\Q$ containing the coefficients of $f$. We denote by $\ovl f$ the homogeneous part of degree $d$ of $f$. The polynomial
$$\ovl f(x_1,\ldots, x_{n-1},1)\not\equiv 0$$
otherwise $\ovl f$ would be divisible by $x_n-1$ which is impossible since $\ovl f$ is a homogeneous polynomial. So there exists 
$$(\mu_1,\ldots, \mu_{n-1})\in \Q^{n-1}$$
such that $c:=\ovl f(\mu_1,\ldots, \mu_{n-1},1)\neq 0$. Then let us denote by $\phi_\mu$ the linear change of coordinates defined by
$$\begin{array}{ccc} x_i & \longmapsto & x_i+\mu_ix_n \text{ for } i<n\\
x_n & \longmapsto & x_n.\end{array}$$
We have that 
$$\phi_\mu(\ovl f)=cx_n^d+h$$
where $h$ is a homogeneous polynomial of degree $d$  belonging to the ideal generated by $x_1,\ldots, x_{n-1}$. So we have that
$$\phi_\mu(f)=cx_n^d+\sum_{j=1}^da_j(x^{n-1})x_n^{d-j}$$
for some polynomials $a_j(x^{n-1})\in\k[x^{n-1}]$.

\begin{rem}
If $g=g_1\ldots g_s$ is a product of polynomials of $\k[x]$, then the linear change of coordinates $\phi_{\mu}$ defined above also satisfies
$$\phi_{\mu}(g_i)=c_ix_n^{d_i}+\sum_{j=1}^da_{i,j}(x^{n-1})x_n^{d-j}$$
for some nonzero constants $c_i\in\k$.

\end{rem}

\begin{rem}\label{coef}
For every $f\in\k[x]$ of degree $d$ where $\k\subset \K$ and $\mu\in\Q^{n-1}$, we have that
$$\phi_{-\mu}\circ \phi_{\mu}(f)=f.$$
Let us define the support of $f=\sum_{\a\in\N^n}f_\a x^\a$ as
$$\Supp(f):=\{\a\in\N^n\mid f_\a\neq 0\}.$$
Let us assume that
$$\phi_{\mu}(f)=\sum f'_\a x^\a$$
for some polynomials $f'_\a\in\k$. Then the coefficient $f_\a$ has the form
$$f_\a=P_\a(\mu_1,\ldots,\mu_{n-1}, f'_\b)$$
for some polynomial $P_\a\in\Z[\mu, f'_\b]$ depending only on the $f'_\b$ with $|\b|=|\a|$.
\end{rem}

\section{Main theorem}

We can state now our main result: 
\begin{thm}\label{main2}
Let $V\subset \K^n$ (resp. $V\subset \P_\K^n$) be an affine (resp. projective) algebraic set, where $\K=\R$ or $\C$. Then there exist an affine (resp. projective) algebraic set $W\subset \K^n$ (resp. $W\subset \P_\K^n$) and a homeomorphism $h:\K^n\lgw \K^n$ (resp. $h:\P_\K^n\lgw \P_\K^n$) such that:
\begin{enumerate}
\item[(i)] The homeomorphism maps $V$ onto $W$,
\item[(ii)] $W$ is defined by polynomial equations with coefficients in $\ovl\Q\cap \K$,
\item[(iii)] The variety $W$ is obtained from $V$ by a Zariski equisingular deformation. In particular the homeomorphism $h$ can be chosen semialgebraic and arc-analytic,
\item[(iv)]  Let $g_1$,\ldots, $g_s$ be the generators of the ideal defining $V$. 
 Let us fix $\e>0$.  Then $W$ can be chosen so that  the ideal defining it is generated by polynomials $g'_1$,\ldots, $g'_s\in\ovl\Q[x]$ such that if we write
$$g_i=\sum_{\a\in\N^n}g_{i,\a}x^\a\text{ and } g_i'=\sum_{\a\in\N^n}g'_{i,\a}x^\a$$
then for every $i$ and $\a$ we have that:
$$|g_{i,\a}-g'_{i,\a}|<\e.$$
\item[(v)] Every polynomial relationship with rational coefficients between the $g_{i,\a}$ will also be satisfied by the $g'_{i,\a}$. 
\end{enumerate}
\end{thm}

\begin{rem}\label{rem:arc-analytic}
 Arc-analytic semialgebraic maps were introduced first in \cite{Ku}. By \cite{B-M} they coincide with the category of blow-Nash maps (i.e. blow-analytic and semialgebraic maps), and thus as blow-analytic homeomorphisms they appeared already in the classification of real singularities in \cite{Kuo}. Arc-analytic semialgebraic maps and homeomorphisms are often used in real algebraic geometry.  They share many properties of regular morphisms but are more flexible. For instance, arc-analytic semialgebraic homeomorphisms preserve the weight fitration and the virtual Poincar\`e polynomial \cite{McP}, but the equivalence relation defined by such homeomorphisms does not have continuous moduli \cite{PP}.  For more, the interested reader may consult the survey  \cite{KP} and the Kucharz-Kurdyka talk on 2018 ICM \cite{ICM}.
\end{rem}

\begin{rem}\label{rem_deg}
We will see in the proof the following: let $\k$ denote the field extension of $\Q$ generated by the coefficients of the $g_i$ and assume that $\k\neq \ovl\Q$ (otherwise there is nothing to prove). By the primitive element theorem, $\k$ is a simple extension of a purely transcendental extension $\L$ of $\Q$, i.e. $\Q\lgw \L$ is purely transcendental and $\L\lgw \k$ is a simple extension of degree $d$. Then the coefficients of the $g_i'$ belong to a finite extension of $\Q$ of degree $\leq d$. \\
In particular if $\Q\lgw \k$ is purely transcendental then $W$ is defined over $\Q$.

\end{rem}

\begin{rem}
In fact (v) implies that 
$$\forall i,\a\ \ \ \ g_{i,\a}= 0\Longrightarrow g'_{i,\a}= 0.$$
If $\e$ is chosen small enough we may even assume that 
$$\forall i,\a\ \ \ \ g_{i,\a}= 0\Longleftrightarrow g'_{i,\a}= 0.$$
\end{rem}

\begin{proof}[Proof of Theorem \ref{main2}]
Let us consider a finite number of polynomials $g_1$,\ldots, $g_s\in\K[x]$ generating the ideal defining $V$ and let us denote by $g_{i,\a}$ their coefficients as written in the theorem. After a linear change of coordinates $\phi_\mu$ with $\mu\in\Q^{n-1}$ as in Section \ref{generic_change} we can assume that
$$
\phi_\mu(g_r)=c_rx_n^{p_r}+\sum_{j=1}^{p_r}b_{n-1,r,j}(x^{n-1})x_n^{p_r-j} =\sum_{\b\in\N^n}a_{n,r,\b}x^\b\ \ \forall r=1,\ldots, s.$$
Moreover by multiplying each $\phi_{\mu}(g_r)$ by $1/c_r$ we can assume that $c_r=1$ for every $r$. We denote by $f_n$ the product of the $\phi_\mu(g_r)$ and by $a_{n}$ the vectors of coefficients $a_{n,r,\b}$. The entries of $a_n$ are polynomial functions in the $g_{i,\a}$ with rational coefficients, let us say
\begin{equation}\label{lin_ch}a_n=A_n(g_{i,\a})\end{equation}
for some $A_n=(A_{n,r,\b})_{r,\b}\in \Q(u_{i,\a})^{N_n}$ for some integer $N_n>0$ and some new indeterminates $u_{i,\a}$.\\
Let $l_n$ be the smallest integer such that
$$\Delta_{n,l_n}(a_{n})\not\equiv 0$$
where $\Delta_{n,l}$ denotes the $l$-th generalized discriminant of $f_n$. In particular we have that
$$\Delta_{n,l}(a_n)\equiv 0 \ \ \forall l<l_n.$$
After a linear change of coordinates $x^{n-1}$ with coefficients in $\Q$ we can write
$\Delta_{n,l_n}(a_{n})=e_{n-1}f_{n-1}$
with
$$f_{n-1}=\sum_{\b\in\N^n}a_{n-1,\b}x^\b=x_{n-1}^{d_{n-1}}+\sum_{j=1}^{d_{n-1}}b_{n-2,j}(x^{n-2})x_{n-1}^{d_{n-1}-j}$$
for  some  constants $e_{n-1}, a_{n-1,\b}\in\k$, with $e_{n-1}\neq 0$, and some polynomials $b_{n-2,j}\in\k[x^{n-2}]$. We denote by $a_{n-1}$ the vector of coefficients $a_{n-1,\b}$. Let $l_{n-1}$ be the smallest integer such that
$$\Delta_{n-1,l_{n-1}}(a_{n-1})\not\equiv 0$$
where $\Delta_{n-1,l}$ denotes the $l$-th generalized discriminant of $f_{n-1}$.\\
We repeat this construction and define a sequence of polynomials $f_j(x^j)$, $j=k_0,\ldots, n-1$ for some $k_0$, such that
$$\Delta_{j+1,l_{j+1}}(a_{j+1})=e_j\left(x_{j}^{d_{j}}+\sum_{k=1}^{d_{j}}b_{j-1,k}(x^{j-1})x_{j}^{d_{j}-k}\right)=e_j\left(\sum_{\b\in\N^n} a_{j,\b}x^\b\right)=e_jf_j$$
is the first non identically zero generalized discriminant  of $f_{j+1}$, where $a_j$ denotes the vectors of coordinates $a_{j,\b}$ in $\k^{N_j}$, i.e. we have that :
\begin{equation}\label{2}\Delta_{j+1,l_{j+1}}(a_{j+1})=e_j\left(\sum_{\b\in\N^n} a_{j,\b}x^\b\right)=e_jf_j \end{equation}
and 
\begin{equation}\label{3}\Delta_{j+1,l}(a_{j+1})\equiv 0\ \ \forall l<l_{j+1},\end{equation}
until we get $f_{k_0}=1$ for some $k_0\geq 1$. 

By  \eqref{2} we see   that the entries of  $a_j$ and  $e_j$ are rational functions in the entries of $a_{j+1}$ for every $j<n$. Thus by \eqref{lin_ch} we see that the entries of the $a_k$ and the $e_j$ are rational functions in the $g_{i,\a}$ with rational coefficients, let us say
$$a_k=A_k(g_{i,\a}),\ \ e_j=E_j(g_{i,\a})$$
for some $A_k\in\Q(u_{i,\a})^{N_k}$ and $E_j\in\Q(u_{i,\a})$ for every $k$ and $j$, where the $u_{i,\a}$ are new indeterminates for every $i$ and $\a$.\\
\\
By Theorem \ref{thm_alg}, there exist a new set of indeterminates $t=(t_1,\ldots, t_r)$, an open domain $\mathcal U\subset\K^r$,  $(\tt_1,\ldots, \tt_r)\in\mathcal U$, polynomials 
$g_{i,\a}(t,z)\in\Q(t)[z]$, for every $i$ and $\a$,  and $z(t)\in\FAGN(\mathcal U)$ such that $g_{i,\a}(\tt,z(\tt))=g_{i,\a}$ for every $i$ and $\a$.
\\

Since the $a_l$ and the $e_j$ are  rational functions  with rational coefficients in the $g_{i,\a}$, the system of equations \eqref{3} is equivalent to a system of polynomial equations with rational coefficients 
\begin{equation}\label{dis}f(g_{i,\a})=0.\end{equation}
By  Theorem \ref{thm_alg} the functions $g_{i,\a}(t,z(t))$  are solutions of  this system of equations.

 Because $e_j(g_{i,\a})=e_j(g_{i,\a}(\tt,z(\tt)))\neq 0$ for every $j$, we have that none of the $e_j(g_{i,\a}(t,z(t)))$ vanishes in a small open ball  $\mathcal V\subset\mathcal U$ centered at $\tt$. 

In particular we have that
\begin{equation}\label{lead_coef}e_j(g_{i,\a}(t,z(t)))\neq 0\ \ \forall j,\ \forall t\in\mathcal V.\end{equation}
For $t\in\mathcal V$ and $ r=1,\ldots, s,$ we define 
\begin{align*}
& g'_r(t,x):=\sum_{\a\in\N^n}g_{i,\a}(t,z(t))x^\a,\\
& F_n(t,x)=\prod_{r=1}^sG_r(t,x), \text{ where} \\
& G_r(t,x)=\sum_{\a\in\N^n}A_{n,r,\a}(g_{i,\a}(t,z(t)))x^\a,  r=1,\ldots, s,
\end{align*}
and for $j=k_0,\ldots, n-1$
$$F_j(t,x)=\sum_{\a\in\N^n}A_{j,\a}(g_{i,\a}(t,z(t)))x^\a.$$
In particular we have that $G_r(\tt,x)=\phi_\mu(g_r(x))$ for every $r$.

Let us denote by $\O_{\mathcal V}$ the ring of $\K$-analytic functions on $\mathcal V$.
In particular we have that $F_j(t,x)$ is a polynomial of $\O_{\mathcal V}[x^i]$, of degree $d_j$ in $x_j$, such that the coefficient of $x_j^{d_j}$ is $e_j(g_{i,\a}(t,z(t)))$ and whose discriminant (seen as a polynomial in $x_j$) is equal to $e_{j-1}(g_{i,\a}(t,z(t)))F_{j-1}(t,x)\in\O_{\mathcal V}[x^{j-1}]$ for every $j$ by \eqref{dis}.\\

Thus the family $(F_j)_j$ satisfies the hypothesis of Theorem \ref{equisingularity} by \eqref{lead_coef}. Hence the algebraic hypersurfaces $X_0:=\{F_n(0,x)=0\}$ and $X_1:=\{F_n(1,x)\}$ are homeomorphic. Moreover the homeomorphism between them maps every  component of $X_0$ defined by $G_r(0,x)=0$ onto the component of $X_1$ defined by $G_r(1,x)=0$. This proves that the algebraic variety defined by $\{\phi_{\mu}(g_1)=\cdots =\phi_{\mu}(g_s)=0\}$ is homeomorphic to the algebraic variety $\{G_1(0,x)=\cdots=G_s(0,x)=0\}$ which is defined by polynomial equations over $\ovl\Q$. Thus $V$ is homeomorphic to $W:=\{g'_1(x)=\cdots=g'_s(x)=0\}$. 
\\
\\
Moreover since $\ovl\Q\cap \K$ is dense in $\K$, we can choose $\q\in\ovl\Q$ as close as we want to $\tt$. In particular by choosing $\q$ close enough to $\tt$ we may assume that (iv)  in Theorem \ref{main2} is satisfied, since $\g$, $t\mapsto z(t)$ and the $g_{i,\a}$ are continuous functions.\\
\\
Finally we have that $z(\q)$ is algebraic of degree $\leq d$ over $\Q$ by Lemma \ref{alg} (see Remark \ref{rem_deg}).\\
Thus Theorem \ref{main2} is proven in the affine case. The projective is proven is the same way by Remark \ref{proj}.
\end{proof}


We can remark that Theorem \ref{main2} (v) implies that several algebraic invariants are preserved by the homeomorphism $h$. For instance we have the following corollary:

\begin{cor}\label{HS}
 For $\e>0$ small enough the Hilbert-Samuel function of $W$ is the same as that of $V$.
\end{cor}

\begin{proof}[Proof of Corollary \ref{HS}]
Let $h_1$, \ldots, $h_m$ be a Gr\"obner basis of the ideal $I$ of $\K[x]$ generated by the $g_i$ with respect to a given monomial order $\preceq$. Let us recall that for $f=\sum_{\a\in\N^n} f_\a x^\a\in\K[x]$, $f\neq 0$, we denote the leading term of $f$ by $\lt(f)=f_{\a_0}x^{\a_0}$ where $\a_0$ is the largest nonzero exponent of the support of $f$ with respect to the monomial order:
$$\Supp(f)=\{\a\in\N^n \mid f_\a\neq 0\}.$$
A Gr\"obner basis of $I$ is computed by considering S-polynomials and divisions (see \cite{CLO} for more details):\\
- the S-polynomial of two nonzero polynomials $f$ and $g$ is  defined as follows: set $\lt(f)=ax^\a$ and $\lt(g)=bx^\b$ with $a$, $b\in\K^*$, and let $x^\d$ be the least common multiple of $x^\a$ and $x^\b$. Then the S-polynomial of $f$ and $g$ is
$$S(f,g):=\dfrac{x^{\d-\a}}{a}f-\dfrac{x^{\d-\b}}{b}g.$$
- the division of a polynomial $f$ by polynomials $f_1$, \ldots, $f_l$ is defined inductively as follows: firstly after some renumbering one assume that $\lt(f_1)\preceq \lt(f_2)\preceq \cdots\preceq \lt(f_l)$. Then we consider the smallest integer $i$ such that $\lt(f)$ is divisible by $\lt(f_i)$. If such a $i$ exists one sets $q_i^{(1)}=\dfrac{\lt(f)}{\lt(f_i)}$ and $q_j^{(1)}=0$ for $j\neq i$ and $r^{(1)}=0$. Otherwise one sets $q^{(1)}_j=0$ for every $j$ and $r^{(1)}=\lt(f)$. We repeat this process by replacing $f$ by $f-\sum_{j=1}^lq^{(1)}_jf_j-r^{(1)}$. After a finite number of steps we obtain a decomposition
$$f=\sum_{j=1}^lq_j f_j+r$$
where none of the elements of $\Supp(r)$ is divisible by any $\lt(f_i)$. The polynomial $r$ is called the remainder of the division of $f$ by the $f_i$.\\
Thus we can make the following remark: \emph{every remainder of the division of some S-polynomial $S(g_i,g_j)$ by $g_1$, \ldots, $g_s$ is a polynomial whose coefficients are rational functions on the coefficients of the $g_k$.}\\
The Buchberger's Algorithm is as follows: we begin with $g_1$,\ldots, $g_s$ the generators of $I$ and we compute all the S-polynomials of every pair of polynomials among the $g_i$. Then we consider the remainders of the divisions of these S-polynomials by the $g_i$. If some remainders are nonzero we add them to the family $\{g_1,\ldots, g_s\}$. Then we repeat the same process that stops after a finite number of setps.\\
By the previous remark  if $h$ denotes the remainder of the division of a S-polynomial $S(g_i,g_j)$ by the $g_k$ then we have a relation of the form
\begin{equation}\label{div}S(g_i,g_j)=\sum_{l=1}^s b_lg_l+h\end{equation}
and the $b_l$ belong to the field extension of $\Q$ generated by the coefficients of the $g_l$. Let $h_\b$ denote the coefficient of $x^\b$ in $h$. Then for those $h_\b$ that are nonzero, Equation \eqref{div} provides an expression of them as rational functions in the $g_{i,\a}$, let us say
\begin{equation}\label{div1}\forall \b\in\Supp(h)\ \ \ \ h_\b=H_\b(g_{i,\a}).\end{equation}
 For those $h_\b$ equal to zero Equation \eqref{div} provides a polynomial relation with rational coefficients between the $g_{i,\a}$:
\begin{equation}\label{div2} \forall \b\notin \Supp(h)\ \ \ \ \ H_\b(g_{i,\a})=0.\end{equation}
And Equation \eqref{div} is equivalent to the systems of equations \eqref{div1} and \eqref{div2}. This remains true if we replace the $g_i$ by polynomials whose coefficients are rational functions in the $g_{i,\a}$ with rational coefficients. Thus the fact that $h_1$,\ldots, $h_m$ is a Gr\"obner basis obtained from the $g_i$ by Buchberger's Algorithm is equivalent to a system of equations
\begin{equation}\label{grb} Q_k(g_{i,\a})=0  \text{ for }k\in E\end{equation}
where $E$ is a (not necessarily finite) set. 
Thus by Theorem \ref{main2} (v) we see that these equations are satisfied by the coefficients of the $g_i'$, hence  the ideal defining $W$ has a Gr\"obner basis $h_1'$, \ldots, $h_m'$ obtained from the $g'_i$ by doing exactly the same steps in Buchberger's Algorithm, and $\Supp(h'_i)\subset\Supp(h_i)$ for every $i$.\\
Moreover  the initial terms of the $h_i$ are rational functions in the $g_{i,\a}$ with rational coefficients, let us say $H_i(g_{i,\a})$ for some rational functions $H_i$. By choosing $\e$ small enough we insure that
$$H_i(g'_{i,\a})\neq 0 \ \ \ \forall i.$$
Thus the leading exponents of the $h'_i$ are equal to those of the $h_i$.  In particular the Hilbert-Samuel function of $W$ is the same as that of $V$, and (iv) in Theorem \ref{main2} is proven.
\end{proof}

\begin{rem}
In fact we have proven that the ideal of leading terms of $I$ is the same as the ideal of leading terms of the ideal defining $W$. So we could have concluded by \cite[Theorem 5.2.6]{GP} for instance.
\end{rem}


\section{Complete algorithm}
\subsection{Settings}\ \\
\textbf{Input: } \begin{enumerate}
\item polynomials $g_1$, \ldots, $g_s\in\K[x]$ whose coefficients $g_{i,\a}$ belong to a finitely generated field extension $\k$ over $\Q$. 
\item a presentation 
$$\k=\Q(\tt_1,\ldots,\tt_r,\z)$$
where the $\tt_i$ are algebraically independent over $\Q$ and $\z$ is finite over $\Q(\tt_1,\ldots, \tt_r)$. 
\item the minimal polynomial $P(z)$ of $\z$ over $\Q(\tt_1,\ldots, \tt_r)$. 
\item the $g_{i,\a}$ are given as rational functions in the $\tt_i$ and $\z$, and the coefficients of $P(z)$ as rational functions in the $\tt_i$.

\item a positive real number $\e$.
\end{enumerate}
Moreover we assume that the $\tt_i$ and $\z$ are computable numbers \cite{Tu}. Let us recall that a real computable number is a number $\x\in\R$ for which there is a Turing machine that computes a sequence $(\q_n)$ of rational numbers such that $|\x-\q_n|<\frac{1}{n}$ for every $n\geq 1$. A computable number is a complex number whose real and imaginary parts are real computable numbers. The set of  computable numbers is an algebraically closed field \cite{Ri}. More precisely if $\tt_1$,\ldots, $\tt_r$ are computable numbers such that for every $i\in\{1,\ldots, r\}$ $(\q_{i,n})_n$ is a sequence of $\Q+i\Q$ computed by a Turing machine with $|\tt_i-\q_{i,n}|<\frac{1}{n}$ for every $n$, and if $\z\in\C$ satisfies $P(\tt,\z)=0$ for some reduced polynomial $P(t,z)\in\Q[t,z]$, then one can effectively find a Turing machine that computes a sequence $(\q_n)_n$ of $\Q+i\Q$ such that 
$|\z-\q_n|<\frac{1}{n}$ for every integer $n$.\\

\textbf{Output: } polynomials $g'_1$,\ldots, $g'_s\in(\ovl\Q\cap\K)[x]$ with the properties:
\begin{enumerate}
\item the pairs $(V(g_i),\K^n)$ and $(V(g'_i),\K^n)$ are homeomorphic, and the homeomorphism is semialgebraic and arc-analytic.
\item the coefficients $g'_{i,\a}$ of $g'_i$ satisfy the properties:
$$g'_{i,\a}\neq 0 \Longleftrightarrow g_{i,\a}\neq 0$$
$$|g'_{i,\a}-g_{i,\a}|<\e \text{ for every } i \text{ and }\a$$
and  every polynomial relation with coefficients in $\Q$ satisfied by the  $g_{i,\a}$ is also satisfied by the $g'_{i,\a}$.
\end{enumerate}

\subsection{Algorithm:} We present here the successive steps of the algorithm.\\
\\
\textbf{(1)} We make a linear change  of coordinates with coefficients in $\Q$, denoted by $\phi_\mu$ with $\mu\in\Q^{n-1}$, such that each of the $g_i$ is a monic polynomial in $x_n$ of degree $\deg(g_i)$.\\
We denote by $f_n$ the product of the $g_i$ after this change of coordinates, and by $a_n$ the vector of the coefficients of the $g_i$  after this change of coordinates (seen as a polynomial in $x_n$).\\
\\
\textbf{(2)} For every $j$ from $n$ to $1$ we do the following:  let us assume that $f_j$ is a polynomial in $x_1$, \ldots, $x_j$ having a nonzero monomial $e_jx_j^{\deg(f_j)}$, $e_j\in\k^*$. We denote by $a_j$ the vector of the coefficients of $f_j$ seen as a polynomial in $x_j$. We consider the generalized discriminants $\Delta_{j,l}$ of $f_j$ with respect to $x_j$, and we denote by $l_j$ the smallest integer such that
$$\Delta_{j,l_j}(a_j)\neq 0.$$
These polynomials $\D_{j,l_j}(a_j)$ can be effectively computed (see \ref{gd}).\\
We perform a linear change of coordinates (with coefficients in $\Q$) in $x_1$, \ldots, $x_{j-1}$ such that $\Delta_{j,l_j}(a_j)$ becomes a unit times a monic polynomial of degree $\deg(\Delta_{j,l_j}(a_j))$ in $x_{j-1}$, and we denote by $f_{j-1}$ this new monic polynomial.
\\
\\
\textbf{(3)} We stop the process once we have that $f_{j-1}$ is a nonzero constant.\\
\\
\textbf{(4)} We consider \eqref{2}  as a system of polynomial equations that will allow us to compute the expression of   the $e_j$ as rational functions in the $g_{i,\a}$.\\
\\
\textbf{(5)} We denote by $P(t,z)$ the monic polynomial of $\Q(t)[z]$ such that $P(\tt,z)=P(z)$ is the minimal polynomial of $\z$ over $\Q(\tt)$. \\
By replacing $\e$ by a smaller positive number we may assume that 
$$d(\tt,\D_P)>2\e$$
where $\D_P$ denotes the discriminant locus of  $P(t,z)$ seen as a polynomial in $z$. This discriminant  is computed as $\Res_z(P,\partial P/\partial z)$. See \cite{DGY} or \cite[Theorem C]{BM} for a practical way of choosing such a $\e$.
\\
\\
\textbf{(6)} Let $\z_1:=\z$, $\z_2$,\ldots , $\z_d$ be the distincts roots of $P(\tt,z)$. These are computable numbers and so we can compute $\eta\in\Q_{>0}$ such that  all the differences between two $\z_i$ are strictly greater than $\eta$.\\
Let $z(t)$ be the root of $P(t,z)$ such that $z(\tt)=\z$. We can write
$$w(t):=z(t+\tt)=\sum_{\a\in\N^r}\z_\a t^\a,\ \ \z_0=\z.$$
 We write 
$$P(t,z)=p_0(t)+p_1(t)z+\cdots+p_{d-1}(t)z^{d-1}+z^d$$
where the $p_i(t)\in\K(t)$. Set 
$$M:=1+\max_{0\leq i\leq d-1}\max_{\xxi\in B(0,2\e)} |p_i(\xxi)|.$$
Then $|\z_\a|\leq \dfrac{M}{(2\e)^{|\a|}}$ for every $\a$ (by the Cauchy bounds for the roots of a monic polynomial since $w(t)$ is convergent on $B(0,2\e)$ by (5)). Let $W_k(t)$ be the homogeneous term of degree $k$ in the Taylor expansion of $w(t)$: $W_k(t)=\sum_{|\a|=k}\z_\a t^\a$. Then
$$\forall \xxi\in B(0,\e)\ \ \ |W_k(\xxi)|\leq \dfrac{M}{2^k}\binom{k+r-1}{r-1}.$$
We have that
$$\forall k\geq r\ \ \ \ \dfrac{M}{2^k}\binom{k+r-1}{r-1}\leq \dfrac{M2^r}{(r-1)!}\dfrac{k^r}{2^k}.$$
Then choose $k_0\geq r$ such that $ \dfrac{k^r}{\sqrt{2}^{k}}\leq 1$ forall $k\geq k_0$. Therefore
$$\forall \xxi\in B(0,\e), \forall k\geq k_0\ \ \ |W_k(\xxi)|\leq \dfrac{M2^r}{(r-1)!}\dfrac{1}{\sqrt{2}^{k}}$$
and
$$\forall \xxi\in B(0,\e), \forall k\geq k_0\ \ \ \left|\sum_{l\geq k}W_l(\xxi)\right|\leq \dfrac{M2^r}{(r-1)!}\frac{\sqrt{2}}{\sqrt 2-1}\dfrac{1}{\sqrt{2}^{k}}=C\dfrac{1}{\sqrt{2}^{k}}.$$
\\
\\
\textbf{(7)} Now we want to  determine the computable number $z(\q)=w(\q-\tt)$ for a given choice of $\q\in B(\tt,\e)\cap(\Q+i\Q)^r$. This number is one of the roots of $P(\q,z)$. These roots are computable numbers and so we can bound from below all the differences between each two of them: let $\d\in\Q_{>0}$ be such a bound.\\
By the previous step one can compute an integer $k$ such that
$$\forall \xxi\in B(0,\e), \ \left|\sum_{l\geq k}W_l(\xxi)\right|\leq C\dfrac{1}{\sqrt{2}^{k}}\leq \dfrac{\d}{2}.$$
So we can distinguish $z(\q)$ from the other roots of $P(\q,z)$.
\\
\\
\textbf{(8)} Choose $\q\in(\Q+i\Q)^r$ such that 
\begin{enumerate}
\item[(i)] $\q$ is not in the discriminant locus of $P(t,z)$ seen as a polynomial in $z$,
\item[(ii)] $\|\q-\tt\|<\e$,
\item[(iii)] $e_j(g_{i,\a}(\q,z(\q)))\neq 0$.

\end{enumerate}
The first condition is insured by the choice of $\e$ in (5).\\
 In order to check that $e_j(g_{i,\a}(\q,z(\q)))\neq 0$, one only has to choose $\q$ such that $\|\tt-\q\|$ is small enough and this can be done effectively. Indeed the $e_j(g_{i,,\a})$ are rational functions in the $t_i$ and $z$, thus we can effectively bound the variations of $e_j$ locally around $\tt$.
 \\
 \\
 \textbf{(9)}  Then we evaluate  the $g_{i,a}(t,z(t)) $
at $(\q,z(\q))$. We denote these values by $ g'_{i,\a}$ and we define
 the polynomials
$$g'_i:=\sum_{\a\in\N^n}g'_{i,\a}x^\a.$$

\subsection{Generalized discriminants}\label{gd}
We follow  Appendix IV of \cite{W}, \cite{BPRbook} or \cite{Roy}. The generalized discriminants, or subdiscriminants, $\D_l$ of a polynomial
$$f=x^d+\sum_{j=1}^db_jx^{d-j}$$
can be defined as follows:\\
Let $\xi_1$,\ldots, $\xi_d$ be the roots of $f$ (with multiplicities), and set $s_i=\sum_{k=1}^d \xi_k^i$ for every $i\in \N$. Then
$$\D_{d+1-l}=\left|\begin{array}{cccc} s_0 & s_1 & \cdots & s_{l-1}\\
s_1 & s_2 & \cdots & s_l \\
\cdots & \cdots & \cdots & \cdots\\
s_{l-1} & s_l & \cdots & s_{2l-2} \end{array}\right|.$$
Thus $\Delta_1$ is the classical discriminant of $f$.  
The polynomials $\D_l$ may be effectively computed in term of the $s_i$, and those can be effectively computed in terms of the $b_i$.  The polynomial $f$ admits exactly 
$k$ distinct complex roots if and only if $\Delta_{1} = \cdots = \Delta_{d-k}=0$ and 
$\Delta_{d-k+1} \ne 0$.



\begin{thebibliography}{00}

\bibitem[BPRbook]{BPRbook} S. Basu, R. Pollack, and M.-F. Roy, Algorithms in real algebraic geometry, vol. 10 of Algorithms and Computation in Mathematics, Springer-Verlag, Berlin, second ed., 2006. 

\bibitem[BM90]{B-M}  E. Bierstone, P. D. Milman,
Arc-analytic functions,
\emph{Invent. math.}  \textbf{101} (1990), 411--424.

\bibitem[BPR17]{BPR} M. Bilski, A. Parusi\'nski, G. Rond, Local topological algebraicity of analytic function germs, \textit{J. Algebraic Geom.}, \textbf{26}, (2017), 177-197.

\bibitem[BCR] {BCR}  J. Bochnak, M. Coste, M.-F. Roy,
\emph{Real Algebraic Geometry}, Springer-Verlag, Berlin Heidelberg
1998.

\bibitem[BM95]{BM}  L. P. Bos, P. D. Milman,  Sobolev-Gagliardo-Nirenberg and Markov type inequalities on subanalytic domains, \textit{Geom. Funct. Anal.}, \textbf{5}, (1995), no. 6, 853-923.

\bibitem[CLO07]{CLO} D. Cox, J. Little, D. O'Shea, Ideals, Varieties, and Algorithms, An introduction to computational algebraic geometry and commutative algebra, Springer, New York, 2007.

\bibitem[DGY96]{DGY} J.-P. Dedieu, X. Gourdon, J.-C. Yakoubsohn, 
Computing the distance from a point to an algebraic hypersurface, The mathematics of numerical analysis (Park City, UT, 1995), 285-293,
Lectures in Appl. Math., 32, Amer. Math. Soc., Providence, RI, 1996. 

\bibitem[GP08]{GP} M. Greuel, G. Pfister, A SINGULAR introduction to commutative algebra, \textit{Springer, Berlin}, 2008.

\bibitem[H80]{H}
R. Hardt, Semi-algebraic local-triviality in semi-algebraic mappings. {\em Amer. J. Math.} 102 (1980), no. 2, 291--302.

\bibitem[Kuo85]{Kuo}  T.-C. Kuo,
On classification of  real singularities, 
      \emph{Invent. math.} \textbf{82} (1985), 257--262.
      
      \bibitem[Ku88]{Ku} K. Kurdyka, 
Ensembles semi-alg\'ebriques
sym\'etriques par arcs, \emph{Math. Ann.}  \textbf{ 282} (1988), 445--462.

\bibitem[KK18]{ICM} W. Kucharz, K. Kurdyka, From continuous ration to regulous functions, Proc. Int. Cong. of Math. -- 2018 Rio de Janeiro, Vol 1 (715-744).

\bibitem [KP08]{KP}  K. Kurdyka, A. Parusi\'nski,
\emph{Arc-symmetric Sets and Arc-analytic Mappings}, dans 
"Arc Spaces and Additive Invariants in Real Algebraic Geometry",  
 Proceedings of Winter School "Real algebraic and Analytic Geometry and 
Motivic Integration", Aussois 2003, eds. M. Coste, K. Kurdyka,
A. Parusi\'nski, 
Panoramas et Synth\`eses 24, 2008, 33--68.


\bibitem[L95]{Lempert95} L. Lempert, Algebraic approximations in analytic geometry, Invent. Math., 121 (1995),
pp. 335--353.

  \bibitem[MP2011]{McP}  C. McCrory, A. Parusi\'nski,
The weight filtration for real algebraic varieties,
\emph{MSRI Publications}, \textbf{18}  "Topology of
Stratified Spaces" (2011), 121--160

\bibitem[Mo84]{Mo} T. Mostowski, Topological equivalence between analytic and algebraic sets, \textit{Bull. Polish Acad. Sci. Math.}, \textbf{32}, (1984), no. 7-8, 393-400. 

\bibitem[PP17]{PP}  L. P\u{a}unescu, A. Parusi\'nski, Arc-wise analytic stratification, Whitney fibering conjecture and Zariski equisingularity, \textit{Adv. Math.}, \textbf{309},  (2017), 254-305. 

\bibitem[RD84]{RD} P.  Ribenboim, L. Van den Dries,,
The absolute Galois group of a rational function field in characteristic zero is a semi-direct product, \textit{Can. Math. Bull.}, \textbf{27}, (1984), 313-315.

\bibitem[Ri54]{Ri} H. G. Rice, Recursive real numbers, \textit{Proc. Amer. Math. Soc.}, \textbf{5}, (1954), 784-791.

\bibitem[Ro]{Ro} G. Rond, Local topological algebraicity with algebraic coefficients of analytic sets or functions, \textit{Algebra Number Theory},  \textbf{12}, (2018), no. 5, 1215-1231.

\bibitem[Roy06]{Roy} M.-F. Roy, Subdiscriminant of symmetric matrices are sums of squares, \textit{Mathematics, Algorithms, Proofs}, volume 05021 of Dagstuhl Seminar Proceedings, Internationales Begegnungs und Forschungszentrum f\"ur Informatik (IBFI), Schloss Dagstuhl, Germany, (2005).

\bibitem[Te90]{Te} B.  Teissier,  Un exemple de classe d'\'equisingularit\'e irrationnelle, \textit{C. R. Acad. Sci. Paris S\'er. I Math.}, \textbf{311}, (1990), no. 2, 111-113.

 \bibitem[Tu36]{Tu} A. M.  Turing,  On Computable Numbers, with an Application to the Entscheidungsproblem, \textit{Proc. London Math. Soc. (2)}, \textbf{42}, (1936), no. 3, 230-265.
  
  \bibitem[Va72]{Va} A. N. Varchenko, Theorems on the topological equisingularity of families of algebraic varieties and families of polynomial mappings, \textit{Math. USSR Izviestija}, \textbf{6}, (1972), 949-1008.


\bibitem[Va73]{Va73} A. N. Varchenko,  
The relation between topological and algebro-geometric equisingularities according to  Zariski. \emph{Funkcional. Anal. Appl.} \textbf{7} (1973),  87--90.

\bibitem[Va75]{VaICM} A. N. Varchenko,   Algebro-geometrical equisingularity and local topological classification of smooth mappings.  Proceedings of the International Congress of Mathematicians (Vancouver, B.C., 1974), Vol. 1, pp. 427--431. Canad. Math. Congress, Montreal, Que., 1975.

  
\bibitem[Wh72]{W} H. Whitney, Complex Analytic Varieties,  Addison-Wesley Publ. Co., Reading, Massachusetts 1972.

\bibitem[Za71]{Za}
O. Zariski, Some open questions in the theory of singularities, \emph{Bull. Amer. Math. Soc.} 
\textbf{77} (1971) 481--491; Oscar Zariski: Collected Papers, Volume IV, MIT Press, 238--248.

\end{thebibliography}
\end{document}